\documentclass[letterpaper, 10 pt, conference]{ieeeconf}  

\IEEEoverridecommandlockouts                              

\usepackage{graphicx}
\usepackage{epsfig}
\usepackage{times}
\usepackage{amsmath}
\usepackage{amssymb}

\usepackage{algorithm} 
\usepackage{algpseudocode} 
\usepackage{amsfonts}
\usepackage{bm}
\usepackage{cite}
\usepackage{commath}
\usepackage{fancyhdr}
\usepackage{float}
\usepackage[T1]{fontenc}
\usepackage[utf8]{inputenc}
\usepackage{xcolor}

\usepackage{soul} 
\usepackage{mathtools}
\mathtoolsset{showonlyrefs}

\let\labelindent\relax
\usepackage{enumitem}
\usepackage{hyperref}

\usepackage{amsthm}
\usepackage{subcaption}

\allowdisplaybreaks

\usepackage{auxiliary}
\usepackage{blindtext}

\renewcommand{\top}{\textit{\footnotesize \texttt{T}}} 

\newtheorem{theorem}{Theorem}
\newtheorem{lemma}[theorem]{Lemma}
\newtheorem{prop}[theorem]{Proposition}

\newtheorem{defn}{Definition}

\newtheorem{rem}{Remark}
\newtheorem{assumption}{Assumption}
\makeatletter
\let\NAT@parse\undefined
\makeatother

\allowdisplaybreaks

\title{\LARGE \bf Intrinsic Decentralized Stochastic Riemannian Optimization \\ on Manifolds with Bounded Sectional Curvature}

\author{Duc Toan Nguyen, C\'esar A. Uribe
	\thanks{Department of Electrical and Computer Engineering, Rice University, Houston, TX, USA. \{duc.toan.nguyen,cauribe\}@rice.edu}
}

\begin{document}
	
	\maketitle
	\thispagestyle{empty}
	\pagestyle{empty}

\begin{abstract}
Decentralized optimization on Riemannian manifolds is foundational for many modern machine learning and signal processing applications in which data are non-Euclidean and generated and processed in a distributed manner. Although intrinsic Riemannian methods exploit manifold geometry without relying on Euclidean embeddings, existing decentralized Riemannian optimization algorithms typically use constant step sizes and therefore converge only to a neighborhood of steady-state error. In this paper, we study the decentralized stochastic Riemannian gradient method in the diminishing step-size regime on manifolds with (possibly positive) bounded sectional curvature. We prove an $\mathcal{O}(1/T)$ bound for the network consensus error and an $\mathcal{O}(\log T/\sqrt{T})$ ergodic bound for the global optimality gap. To the best of our knowledge, this is the first exact, non-asymptotic optimality-gap guarantee for an intrinsic decentralized stochastic Riemannian method in the geodesically convex setting. Furthermore, the diminishing step-size schedule allows substantially larger initial gradient steps than fixed-step baselines, leading to better performance in practice. We illustrate this on the problem of distributed PCA over a Grassmann manifold.
\end{abstract}


\section{Introduction}
We consider the following optimization over a geodesically convex subset $\mathcal{X}$ of a $d$-dimensional Riemannian manifold $\mathcal{M}$:
\begin{equation}\label{prob:centralized-stochastic}
    \min_{x\in\mathcal{X}} \frac{1}{n}\sum_{i=1}^n f_i(x),
\end{equation}
where each local objective is $f_i(x):=\mathbb{E}_{\xi_i\sim D_i}[F_i(x,\xi_i)]$ and is assumed to be geodesically convex, and assume a minimizer $x_\ast \in \arg\min_{x\in\mathcal{X}} \frac{1}{n}\sum_{i=1}^n f_i(x)$ exists. In the empirical setting, each $f_i$ takes the form $f_i(x):=\frac{1}{m_i}\sum_{j=1}^{m_i} F_i(x,\xi_i^j)$. Our goal is to solve \eqref{prob:centralized-stochastic} in a decentralized manner over a network of $n$ agents, where agent $i$ has access only to its own objective $f_i$. To this end, we introduce local copies $(x_1,\dots,x_n)\in\mathcal{X}^n$ and look for solutions to the lifted function $f(\mathbf{x}) \triangleq ({1}/{n})\sum_{i=1}^n f_i(x_i)$ with $x_1=x_2=\cdots = x_n$. In particular, if $x_\ast$ solves \eqref{prob:centralized-stochastic}, then $\mathbf{x}_\ast:=(x_\ast,\dots,x_\ast)$ is a solution of the lifted problem. This formulation is well-suited to applications where data are distributed across agents, since each agent can evaluate its own cost function locally, improving scalability, preserving privacy, and avoiding centralized data aggregation. Problems of this form arise in control, machine learning, and signal processing, including distributed PCA \cite{zhang2016first, wang2024icassp}, Gaussian mixture models \cite{hosseini2015matrix, wang2025riemannian}, and matrix completion \cite{boumal2011rtrmc}.

The literature on decentralized optimization on Riemannian manifolds can be broadly divided into extrinsic and intrinsic approaches. Extrinsic methods operate on submanifolds embedded in Euclidean spaces \cite{sarlette2009consensus, chen2021decentralized, chen2024decentralized, deng2025decentralized, zhao2026distributed}. They achieve network agreement through standard Euclidean consensus and the induced arithmetic mean, followed by projection or retraction steps to pull the averaged states back onto the manifold. However, this dependency restricts their applicability to specific manifolds such as the Stiefel manifold. In contrast, intrinsic methods operate directly on the manifold by exploiting its geometric structure. These approaches rely on tools such as the Fréchet mean, geodesics, Riemannian gradients, and exponential maps, enabling their application to a broader class of manifolds, including Grassmann manifolds \cite{edelman1998geometry} and Bures-Wasserstein manifolds \cite{altschuler2021averaging}. 

Tron et al. \cite{tron2012riemannian} pioneered the use of intrinsic methods to solve consensus problems on manifolds with bounded curvature. Subsequently, \cite{shah2017distributed} extended gradient tracking to the Riemannian setting, achieving asymptotic convergence for convex functions. More recently, \cite{wang2025riemannian, wang2025distributed} generalized diffusion adaptation methods to both convex and non-convex settings; however, their convergence bounds are inexact. Because they rely on fixed step sizes, the algorithms converge only to a steady-state neighborhood of the optimizer. Moreover, their analysis requires geodesic convexity of the squared distance function, an assumption that fails on some manifolds, such as the Bures--Wasserstein manifold \cite{altschuler2021averaging}. Chen and Sun~\cite{chensun2024doro} studied decentralized online geodesically convex optimization on Hadamard manifolds and established dynamic-regret guarantees using both weighted Fr\'echet-mean consensus and a simplified intrinsic consensus step. Sahinoglu and Shahrampour~\cite{sahinoglu2025decentralized} then extended decentralized online Riemannian optimization beyond Hadamard manifolds to spaces with possibly positive curvature via a curvature-aware consensus analysis. 

In this work, we build upon the intrinsic decentralized Riemannian optimization methods in \cite{wang2025riemannian, sahinoglu2025decentralized} by introducing a diminishing gradient step-size. To the best of our knowledge, this is the first exact, non-asymptotic optimality-gap guarantee for an intrinsic decentralized stochastic Riemannian diffusion-type method in the static geodesically convex setting. We additionally prove a non-asymptotic $\mathcal{O}(1/T)$ bound for the network consensus error. By employing a diminishing step-size schedule, our method accommodates a much larger initial step size than fixed-step baselines. Numerical experiments show faster convergence for distributed PCA on the Grassmann manifold.

This paper is organized as follows: Section~\ref{sec:Algorithm} introduces and analyzes our Decentralized Riemannian Optimization algorithm with diminishing step sizes; Section~\ref{sec:Convergence-Analysis} presents the full proofs for the convergence of network consensus (Theorem~\ref{thm:stochastic-consensus}) and the global optimality gap (Theorem~\ref{thm:function-error}); Section~\ref{sec:numerics} includes the numerical experiment with distributed PCA; and Section~\ref{sec:Discussion} contains conclusions and future work.

\vspace{-0.2cm}
\section{Algorithm and Results
}\label{sec:Algorithm}
In this work, we study the decentralized Riemannian diffusion algorithm introduced in \cite{wang2025distributed,sahinoglu2025decentralized} under a diminishing gradient step-size regime. Specifically, we consider the case where the gradient step-size $\eta_t$ varies with the iteration index $t$, while the consensus step-size remains fixed at $s>0$. Let $W=[w_{ij}]$ denote the graph mixing matrix. At iteration $t$, each agent $i$ performs the following two updates:
\begin{subequations}\label{eqn:drd-updates}
\begin{align}
\quad
y_i^{t+1} &= \exp_{x_i^t}\bigl(-\eta_t\,\widehat{\grad} f_i(x_i^t)\bigr), \label{eqn:gradient-step}\\
\quad
x_i^{t+1} &= \exp_{y_i^{t+1}}\biggl(s\sum_{j=1}^n w_{ij}\log_{y_i^{t+1}}(y_j^{t+1})\biggr). \label{eqn:consensus-step}
\end{align}
\end{subequations}
Step \eqref{eqn:gradient-step} is a local stochastic Riemannian gradient update, where agent $i$ moves from $x_i^t$ along the geodesic in the direction opposite to the unbiased stochastic gradient estimator $\widehat{\grad} f_i(x_i^t)$, scaled by the step-size $\eta_t$. This produces the intermediate iterate $y_i^{t+1}$. Step \eqref{eqn:consensus-step} is a manifold-valued consensus update: agent $i$ maps its neighbors' intermediate iterates to the tangent space at $y_i^{t+1}$ through the logarithmic map, aggregates them using the weights $w_{ij}$, and maps the resulting weighted tangent vector back to the manifold via the exponential map. The parameter $s$ controls the strength of this consensus correction. Before turning to the convergence analysis, we introduce the main definitions and assumptions used throughout the paper.

\begin{defn}[Geodesically convex function] A function $F: \mathcal{X} \to \mathbb{R}$ is said to be geodesically convex (g-convex) if for all $x,y \in \mathcal{X}$, $F(y) \geq F(x) + \langle \mathrm{grad} F(x), \log_x(
y)\rangle.$
\end{defn}

\begin{defn}[Fréchet mean]
Given \(\{x_i\}_{i=1}^n \subset \mathcal{X}\), define $
\bar{x} \in \arg\min_{x \in \mathcal{X}} \varphi(x)$ and $\varphi(x):=({1}/{n})\sum_{i=1}^n d^2(x_i,x)$. If the minimizer exists, the Fréchet variance is defined by $\mathrm{Var}(\{x_i\}_{i=1}^n):=\varphi(\bar{x})$. If the Fr\'echet mean is not unique, we fix an arbitrary measurable selection and denote it by \(\bar{x}\). 
\end{defn}

\begin{assumption}[Network topology]\label{assum:Network-topology}
The network is connected and the communication matrix
$W=[w_{ij}] \in \mathbb{R}^{n\times n}$ is symmetric, doubly stochastic,
entrywise nonnegative, and has positive diagonal entries. Moreover, $w_{ij}=0$ whenever agents $i$ and $j$ are not connected. Denote by $\sigma_2(W)$ the second largest singular value of $W$.
Under these conditions, $\sigma_2(W)\in[0,1)$.
\end{assumption}

\begin{assumption}[Manifold]\label{assum:manifold}
We assume that
\begin{itemize}[leftmargin=2.0em,left=-1pt,itemsep=1pt, parsep=-1pt, topsep=-0pt, partopsep=-1pt]
\item The set $\mathcal{X}\subset \mathcal{M}$ is geodesically convex.
    \item The sectional curvature $K$ inside $\mathcal{X}$ is bounded from below and above, $K_{\min}~\leq~K ~\leq~K_{\max}$.
    \item The diameter of set $\mathcal{X}$ is bounded by $D< \mathrm{inj}(\mathcal{X})$. If $K_{\max} > 0$, we further assume that $D < {\pi}/{(2\sqrt{K_{\max})}}$, where the injectivity radius $\mathrm{inj}(\mathcal{X}) \triangleq \inf_{x\in\mathcal{X}} \mathrm{inj}_{\mathcal M}(x)$ ensures the unique existence of the logarithmic map $\log_x(y)$ for any $x, y \in \mathcal{X}$.
\end{itemize}
Moreover, we define the geometric constants $C_1$, $C_2$, $C_3$, and $C_4$ to be the same as those in \cite{sahinoglu2025decentralized}.
\end{assumption}

\begin{rem}\label{rem:Frechet-mean-exist}
Under Assumption~\ref{assum:manifold}, and following \cite[Theorem 3.4.2]{wintraecken2015ambient}, the Fréchet mean $\bar{x}$ of any $n$ points $\{ x_i \}_{i=1}^n \subset \mathcal{X}$ exists uniquely and stays inside $\mathcal{X}$.
\end{rem}

\begin{assumption}[Bounded gradients]\label{assum:bounded-grad}
Each \(f_i\) is continuously differentiable on \(\mathcal{X}\), and there exists
\(\delta>0\) such that $\|\grad f_i(x)\|\le \delta$, $\forall x\in\mathcal{X},\ \forall i\in\{1,\dots,n\}$. Consequently, each \(f_i\) is \(\delta\)-Lipschitz on \(\mathcal{X}\).
\end{assumption}

\begin{assumption}[Unbiased estimators] \label{assum:unbiased-estimator}
Denote by $\mathcal{F}_t$ the filtration generated by the iterates $\{\mathbf{x}^k\}_{k=1}^t$ generated by \eqref{eqn:gradient-step} and \eqref{eqn:consensus-step}. We assume that, for each $i$ and $t$,
\begin{itemize}[leftmargin=2.0em,left=-1pt,itemsep=1pt, parsep=-1pt, topsep=-0pt, partopsep=-1pt]
    \item $\mathbb{E}\left[\widehat{\mathrm{grad}} f_i(x_i^t)\mid \mathcal{F}_t\right] = \mathrm{grad} f_i(x_i^t)$,
    \item $\mathbb{E}\left[\left\lVert \widehat{\mathrm{grad}} f_i(x_i^t)-\mathrm{grad} f_i(x_i^t)\right\rVert^2 \mid \mathcal{F}_t\right] \le \sigma^2$,
    \item There exists $G>0$ such that $\left\lVert \widehat{\grad} f_i(x_i^t)\right\rVert \le G$ a.s. 
\end{itemize}
\end{assumption}
Under these assumptions, we can now state the consensus convergence result.
\begin{theorem}\label{thm:stochastic-consensus}
Let Assumptions~\ref{assum:Network-topology}, \ref{assum:manifold},
\ref{assum:bounded-grad}, and \ref{assum:unbiased-estimator} hold. Let
$\{x_i^t\}_{t=1}^{T+1}$ and $\{y_i^{t}\}_{t=1}^{T+1}$ be generated by \eqref{eqn:gradient-step}--\eqref{eqn:consensus-step}
with \(x_i^1=x^1\in\mathcal{X}\) for all \(i\), $\eta_t={\eta_0}/{\sqrt{t}}$, with $\eta_0:=\min\left\{1,{D}/{G}\right\}$, and \(s=C_2/(2C_1)\). 
Assume that all iterates $x_i^t, y_i^{t} \in \mathcal{X}$. Then, for each $t\geq 0$,
$$ \sum_{i=1}^n \mathbb{E}\left[ d^2(x_i^t,\bar{x}^t)\right]
\le {\eta_0^2\,C(\xi)\,nB}/{t}
= \mathcal{O}\left({1}/{t}\right),
$$
where $C(\xi){=}{(1{+}\xi^2)}/{\xi^4}$, $B{=}(1{-}\xi)\bigl(2C_1(\sigma^2{+}\delta^2){+}\xi^{-1}\delta^2\bigr)$, and $\xi={C_2^3(1-\sigma_2(W))}/{(4C_1(1+C_4D^2)^2)}$.
\end{theorem}

\begin{rem}
The assumption that all iterates stay inside the domain $\mathcal{X}$ is a standard regularity condition in Riemannian optimization \cite{zhang2018estimate, kim2022accelerated, wang2025riemannian}. One can remove the assumption of $y_i^{t+1} \in \mathcal{X}$ if a costly projection $\mathcal{P}_{\mathcal{X}}$ is added to Step~\eqref{eqn:gradient-step}~\cite{zhang2016first, sahinoglu2025decentralized}. In~\cite{chen2024decentralized,sahinoglu2025decentralized} a projection for Step~\eqref{eqn:consensus-step} is not explicit, however, assuming every $y_j^{t+1}\in\mathcal{X}$ does not guarantee that $x_i^{t+1} \in \mathcal{X}$ on positively curved manifolds. 
\end{rem}

We next state the exact sublinear convergence rate for the global optimality gap.
\begin{theorem}\label{thm:function-error}
Let Assumptions~\ref{assum:Network-topology}, \ref{assum:manifold},
\ref{assum:bounded-grad}, and \ref{assum:unbiased-estimator} hold. Assume each \(f_i\) is geodesically
convex on \(\mathcal{X}\). Let
$\{x_i^t\}_{t=1}^{T+1}$ and $\{y_i^{t}\}_{t=1}^{T+1}$ be generated by \eqref{eqn:gradient-step}--\eqref{eqn:consensus-step}
with \(x_i^1=x^1\in\mathcal{X}\) for all \(i\), $\eta_t={\eta_0}/{\sqrt{t}}$, with $\eta_0:=\min\left\{1,{D}/{G}\right\}$ for $t \in \{1,\ldots,T \}$, and \(s=C_2/(2C_1)\). Assume that all iterates $x_i^t, y_i^{t} \in \mathcal{X}$. Then, for every horizon $T\geq 1$,
\begin{align*}
&\frac{\sum_{t=1}^T \eta_t\,\mathbb{E}[f(\bar{\mathbf{x}}^t)-f(\mathbf{x}_\ast)]}
{\sum_{t=1}^T \eta_t}
\le
\frac{1}{2n\eta_0\sqrt{T}}\sum_{i=1}^n \mathbb{E}[d^2(x_i^1,x_\ast)] \\
&\quad+
\eta_0\Bigl(\delta\sqrt{C(\xi)B}+C_1(\sigma^2+\delta^2)\Bigr)
\frac{1+\log T}{\sqrt{T}},
\end{align*}
where $\bar{\mathbf{x}}^t {\triangleq} (\bar{x}^t,\dots,\bar{x}^t)$, and $\xi$, $C(\xi)$, and $B$ from Theorem~\ref{thm:stochastic-consensus}.
\end{theorem}
The proofs for these theorems are presented in Section~\ref{sec:Convergence-Analysis}.
\vspace{-0.5cm}
\begin{rem}The sublinear convergence rate $\mathcal{O}(\log T/\sqrt{T})$ nearly matches the rate of first-order centralized Riemannian methods \cite{zhang2016first} and the rate of decentralized stochastic optimization in Euclidean space \cite{sirb2018decentralized}, all with diminishing step sizes.
\end{rem}
\section{Convergence Analysis
}\label{sec:Convergence-Analysis}
In this section, we present detailed proofs of the two main convergence theorems for our proposed algorithm. First, we have two foundational geometric lemmas.
\begin{lemma}[{\cite[Lemma 2]{wang2025distributed}}] \label{lemma:cosine-law}
Let $a,b,c \in \mathcal{X}$ and let Assumption~\ref{assum:manifold} hold. Then,
\begin{equation}
\begin{aligned}
    d^2(a,c) &\leq C_1 d^2(b,c) +  d^2(a,b) - 2 \langle \log_b(a), \log_b(c) \rangle, \\
    d^2(a,c) &\geq C_2 d^2(b,c) +  d^2(a,b) - 2 \langle \log_b(a), \log_b(c) \rangle.
\end{aligned}
\end{equation}
\end{lemma}

\begin{lemma}[\protect{\cite[Lemma 4]{sun2019escaping}}]\label{lemma:bound-diff-log}
Let $x, y, z \in \mathcal{X}$ with the distance between each pair of points not being greater than $D$. Then, under Assumption~\ref{assum:manifold} we have
\begin{align}
    (1+C_3D^2)^{-1} d(y,z)& \leq \lVert \log_{x}(y) - \log_{x}(z) \rVert \\
    &\leq (1+C_4D^2)d(y,z).
\end{align}
\end{lemma}
Now, we show that the iterates $x_i^t$ and $y_i^t$ from our algorithm exist, for every $i$ and $t$.
\begin{lemma}\label{lemma:stay-in-X}
Let Assumptions~\ref{assum:Network-topology}, \ref{assum:manifold}, and \ref{assum:unbiased-estimator} hold. Suppose that at iteration $t$ we have $x_i^t\in\mathcal{X}$ and
$y_i^{t+1}\in\mathcal{X}$ for all $i$. Consider \eqref{eqn:gradient-step}--\eqref{eqn:consensus-step} with $\eta_t=\eta_0/\sqrt{t}$, where $\eta_0\le D/G$, and
$s=C_2/(2C_1)$. Then the logarithmic and exponential maps appearing in
\eqref{eqn:gradient-step} and \eqref{eqn:consensus-step} are well-defined.
\end{lemma}

\begin{proof}
Let \(\hat g_i^t:=\widehat{\mathrm{grad}} f_i(x_i^t)\). Since
\(\eta_t=\eta_0/\sqrt{t}\) with \(\eta_0=\min\{1,D/G\}\), Assumption~\ref{assum:unbiased-estimator}~(iii) gives
\begin{align*}
\|-\eta_t \hat g_i^t\|
\le \eta_t G
= {\eta_0 G}/{\sqrt{t}}
\le \eta_0 G
\le D
< \mathrm{inj}(\mathcal{X}).
\end{align*}
Hence, the gradient update vector lies inside the injectivity domain at
\(x_i^t\), and thus $y_i^{t+1}$ in step \eqref{eqn:gradient-step} are well defined. Next, we analyze the consensus step \eqref{eqn:consensus-step}. By the definitions of \(C_1\) and \(C_2\) \cite{sahinoglu2025decentralized},
we have \(C_1\ge 1\) and \(0<C_2\le 1\), so
\(0<s=C_2/(2C_1)\le 1/2\). Hence
\begin{align*}
\left\|s\sum_{j=1}^n w_{ij}\log_{y_i^{t+1}}(y_j^{t+1})\right\|
&\le s\sum_{j=1}^n w_{ij} d(y_i^{t+1},y_j^{t+1}) \\
&\le sD
< \mathrm{inj}(\mathcal{X}),
\end{align*}
and therefore $x_i^{t+1}$ in step \eqref{eqn:consensus-step} are well defined.

\end{proof}

Next, we will show the convergence of the consensus of our algorithm. We start by analyzing the relationships between three quantities: $\mathrm{Var}(\{y_i^t\}_{i=1}^n)$, $\sum_{i=1}^n \sum_{j=1}^n w_{ij} d^2(y_i^t,y_j^t)$, and $\mathrm{Var}(\{x_i^t\}_{i=1}^n)$.

\begin{lemma}\label{lemma:y_variance}
Let Assumptions~\ref{assum:Network-topology}, \ref{assum:manifold}, and \ref{assum:unbiased-estimator} hold. At iteration $t$,
\begin{align*}
    &\frac{1}{4n(1+C_3D^2)^2}\sum_{i=1}^n \sum_{j=1}^n w_{ij} d^2(y_i^t,y_j^t) \underset{\mathrm{(i)}}{\leq} \mathrm{Var}(\{ y_i^t \}_{i=1}^n), \\
    &  \frac{(1+C_4D^2)^2}{2n(1-\sigma_2(W))}
\sum_{i=1}^n\sum_{j=1}^n w_{ij} d^2(y_i^t,y_j^t) \underset{\mathrm{(ii)}}{\geq} \mathrm{Var}(\{ y_i^t \}_{i=1}^n).
\end{align*}
\end{lemma}
\begin{proof}
We first show inequality (i). Applying Lemma~\ref{lemma:bound-diff-log} for each $y_i^t, y_j^t$, and $\bar{y}^t$ in $\mathcal{X}$, we have
$$(1+C_3D^2)^{-2}d^2(y_i^t,y_j^t) \leq \lVert \log_{\bar{y}^t}(y_i^t) - \log_{\bar{y}^t}(y_j^t) \rVert^2.$$
Since $W$ is doubly stochastic, by summing all distances of pairs, we have
\begin{align*}
   &(1+C_3D^2)^{-2}\sum_{i=1}^n \sum_{j=1}^n w_{ij}d^2(y_i^t,y_j^t) \\
   &\leq 2\sum_{i=1}^n \lVert \log_{\bar{y}^t}(y_i^t) \rVert^2 - 2\sum_{i=1}^n\sum_{j=1}^n w_{ij} \langle \log_{\bar{y}^t}(y_i^t), \log_{\bar{y}^t}(y_j^t) \rangle \\
   &= 2\sum_{i=1}^n d^2(\bar{y}^t,y_i^t) - 2\sum_{i=1}^n\sum_{j=1}^n w_{ij} \langle \log_{\bar{y}^t}(y_i^t), \log_{\bar{y}^t}(y_j^t) \rangle.
\end{align*}
By fixing an orthonormal basis for $T_{\bar{y}^t}\mathcal{M}$, we can isometrically map the tangent vectors $\log_{\bar{y}^t}(y_i^t)$ to $\mathbb{R}^d$ and stack their coordinates as columns of a matrix $V \in \mathbb{R}^{d \times n}$. Since $W$ is symmetric and doubly stochastic ($\lambda_{\min}(W) \geq -1$), we have
\begin{equation*}
\begin{aligned}
    & \sum_{i=1}^n \sum_{j=1}^n w_{ij} \langle \log_{\bar{y}^t}(y_i^t), \log_{\bar{y}^t}(y_j^t) \rangle = \text{Tr}(V W V^\top) \\
    &\geq \lambda_{\min}(W) \text{Tr}(V V^\top) \geq -1 \cdot \text{Tr}(V V^\top) = - \sum_{i=1}^n d^2(\bar{y}^t, y_i^t).
\end{aligned}
\end{equation*}
\vspace{-0.3cm}
Thus, we have
\begin{align*}
   &(1+C_3D^2)^{-2}\sum_{i=1}^n \sum_{j=1}^n w_{ij}d^2(y_i^t,y_j^t) \\
   &\leq 2\sum_{i=1}^n d^2(\bar{y}^t,y_i^t) - 2\sum_{i=1}^n\sum_{j=1}^n w_{ij} \langle \log_{\bar{y}^t}(y_i^t), \log_{\bar{y}^t}(y_j^t) \rangle \\
   &\leq 2\sum_{i=1}^n d^2(\bar{y}^t,y_i^t) + 2\sum_{i=1}^n d^2(\bar{y}^t,y_i^t) = 4\sum_{i=1}^n d^2(\bar{y}^t,y_i^t).
\end{align*}
Then, we have inequality (i). Moreover, the proof of inequality (ii) is shown in \cite[Lemma III.1]{sahinoglu2025decentralized}.

\end{proof}

\begin{prop}[\protect{\cite[Theorem III.2]{sahinoglu2025decentralized}}] \label{prop:var_x_var_y_alg_1}
Let Assumptions~\ref{assum:Network-topology}, \ref{assum:manifold}, and \ref{assum:unbiased-estimator} hold. Consider the consensus step \eqref{eqn:consensus-step}. By selecting the step size $s = (2C_1)^{-1}C_2$, we have
$$\mathrm{Var}(\{ x_i^{t+1} \}_{i=1}^n) \leq \rho_1 \mathrm{Var}(\{ y_i^{t+1} \}_{i=1}^n),$$
where $\rho_1 = 1 - {C_2^{3}(1-\sigma_2(W))}/{(4C_1(1+C_4D^2)^2)}$.
\end{prop}
Now, we present the proof for the consensus convergence.
\begin{proof}[\textbf{Proof for Theorem~\ref{thm:stochastic-consensus}}]
Let \(\hat g_i^t := \widehat{\mathrm{grad}} f_i(x_i^t)\) and
\(g_i^t := \mathrm{grad} f_i(x_i^t)\). Then,
\begin{align}
&\mathbb{E}[d^2(x_i^t,y_i^{t+1})]
= \mathbb{E}\left[\|\log_{x_i^t}(y_i^{t+1})\|^2\right] \notag= \eta_t^2 \mathbb{E}\left[\|\hat g_i^t\|^2\right] \notag\\
&\le 2\eta_t^2 \mathbb{E}\left[\|\hat g_i^t-g_i^t\|^2+\|g_i^t\|^2\right] \notag\\
&\le 2\eta_t^2 \mathbb{E}\left[
\mathbb{E}\left[\|\hat g_i^t{-}g_i^t\|^2 \mid \mathcal{F}_t\right]{+}\|g_i^t\|^2
\right] {\le} 2\eta_t^2(\sigma^2{+}\delta^2).
\label{ineq:stochastic-bounded-gradient}
\end{align}

Next, applying Lemma~\ref{lemma:cosine-law} with
\(a=\bar{x}^t\), \(b=x_i^t\), and \(c=y_i^{t+1}\), we obtain
\begin{align*}
&d^2(\bar{x}^t,y_i^{t+1})
\le C_1 d^2(x_i^t,y_i^{t+1}) + d^2(\bar{x}^t,x_i^t) \\
&\quad - 2\langle \log_{x_i^t}(\bar{x}^t),\log_{x_i^t}(y_i^{t+1}) \rangle \\
&= C_1 d^2(x_i^t,y_i^{t+1}) + d^2(\bar{x}^t,x_i^t)  + 2\eta_t \langle \log_{x_i^t}(\bar{x}^t), \hat g_i^t \rangle.
\end{align*}
Taking expectations, using that \(\log_{x_i^t}(\bar{x}^t)\) is
\(\mathcal{F}_t\)-measurable, and applying Young's inequality
\(2\langle a,b\rangle \le \xi^{-1}\|a\|^2+\xi\|b\|^2\),
\begin{align}
&\mathbb{E}[d^2(\bar{x}^t,y_i^{t+1})]
\le C_1 \mathbb{E}[d^2(x_i^t,y_i^{t+1})]
+ \mathbb{E}[d^2(\bar{x}^t,x_i^t)] \notag\\
&\quad + 2\eta_t \mathbb{E}\left[
\left\langle
\mathbb{E}[\hat g_i^t \mid \mathcal{F}_t],
\log_{x_i^t}(\bar{x}^t)
\right\rangle
\right] \notag\\
&= C_1 \mathbb{E}[d^2(x_i^t,y_i^{t+1})]
+ \mathbb{E}[d^2(\bar{x}^t,x_i^t)] \notag\\
&\quad + 2\eta_t \mathbb{E}\left[
\langle g_i^t,\log_{x_i^t}(\bar{x}^t)\rangle
\right] \notag\\
&\le C_1 \mathbb{E}[d^2(x_i^t,y_i^{t+1})]
+ (1+\xi)\mathbb{E}[d^2(x_i^t,\bar{x}^t)] \notag+ \eta_t^2 \xi^{-1}\delta^2 \notag\\
&\le (1+\xi)\mathbb{E}[d^2(x_i^t,\bar{x}^t)]
+ \eta_t^2\bigl(2C_1(\sigma^2+\delta^2)+\xi^{-1}\delta^2\bigr).
\label{ineq:unbiased-grad-product}
\end{align}

Now let
\(A_t:=\sum_{i=1}^n \mathbb{E}[d^2(x_i^t,\bar{x}^t)]\).
Using Proposition~\ref{prop:var_x_var_y_alg_1}, the Fr\'echet-mean property
of \(\bar{y}^{t+1}\), and \eqref{ineq:unbiased-grad-product}, we have
\begin{align}
A_{t+1}
&= \sum_{i=1}^n \mathbb{E}[d^2(x_i^{t+1},\bar{x}^{t+1})] \notag\le \rho_1 \sum_{i=1}^n \mathbb{E}[d^2(y_i^{t+1},\bar{y}^{t+1})] \notag\\
&\le \rho_1 \sum_{i=1}^n \mathbb{E}[d^2(y_i^{t+1},\bar{x}^{t})] \notag\\
&\le \rho_1(1+\xi)A_t
+ \rho_1 \eta_t^2 n\bigl(2C_1(\sigma^2+\delta^2)+\xi^{-1}\delta^2\bigr)
\notag\\
&= \rho_2 A_t + \eta_t^2 n B = \rho_2 A_t + {\eta_0^2 n B}/{t},
\label{ineq:stochastic-consensus-inequality}
\end{align}
where \(\rho_1=1-\xi\), \(\rho_2=(1-\xi)(1+\xi)=1-\xi^2\), and
\(B=(1-\xi)\bigl(2C_1(\sigma^2+\delta^2)+\xi^{-1}\delta^2\bigr)\).
Since \(0<\xi<1\), we have \(0<\rho_2<1\). By the common initialization,
\(A_1=0\). Unrolling \eqref{ineq:stochastic-consensus-inequality} gives
\begin{align*}
A_t
&\le \eta_0^2 nB \sum_{k=1}^{t-1} \rho_2^{\,t-1-k}\frac{1}{k}
= \frac{\eta_0^2 nB}{t}\sum_{k=1}^{t-1}\rho_2^{\,t-1-k}\frac{t}{k}.
\end{align*}
For \(1\le k\le t-1\), the inequality \((t-k)(k-1)\ge 0\) implies
\(t\le k(t-k+1)\), hence \(t/k \le t-k+1\). Therefore, with the change of
index \(j=t-1-k\),
\begin{align*}
A_t
&\le \frac{\eta_0^2 nB}{t}\sum_{k=1}^{t-1}\rho_2^{\,t-1-k}(t-k+1) \\
&= \frac{\eta_0^2 nB}{t}\sum_{j=0}^{t-2}\rho_2^{\,j}(j+2)\le \frac{\eta_0^2 nB}{t}
\left(\sum_{j=0}^{\infty} j\rho_2^{\,j} + 2\sum_{j=0}^{\infty}\rho_2^{\,j}\right) \\
&= \frac{\eta_0^2 nB}{t}
\left(\frac{\rho_2}{(1-\rho_2)^2}+\frac{2}{1-\rho_2}\right).
\end{align*}
Since \(1-\rho_2=\xi^2\), the term in parentheses equals
\((1+\xi^2)/\xi^4 = C(\xi)\). Hence $A_t
\le  {\eta_0^2 C(\xi)\,nB}/{t}$.

\end{proof}


Finally, we present the proof for Theorem~\ref{thm:function-error}.

\begin{proof}[\textbf{Proof for Theorem~\ref{thm:function-error}}]
Applying Lemma~\ref{lemma:cosine-law} with
\(a=x_\ast\), \(b=x_i^t\), and \(c=y_i^{t+1}\), we get
\begin{align*}
d^2(y_i^{t+1},x_\ast)
&\le C_1 d^2(x_i^t,y_i^{t+1}) + d^2(x_i^t,x_\ast) \\
&\quad - 2\langle \log_{x_i^t}(y_i^{t+1}),\log_{x_i^t}(x_\ast)\rangle.
\end{align*}
Rearranging and using
\(\log_{x_i^t}(y_i^{t+1})=-\eta_t \hat g_i^t\),
\begin{align*}
-\eta_t \langle \hat g_i^t,\log_{x_i^t}(x_\ast)\rangle
&\le \frac{1}{2}\Bigl(d^2(x_i^t,x_\ast)-d^2(y_i^{t+1},x_\ast)\Bigr) \\
&\quad + \frac{C_1}{2}d^2(x_i^t,y_i^{t+1}).
\end{align*}
Taking expectations and using \eqref{ineq:stochastic-bounded-gradient},
\begin{align}
-\eta_t \mathbb{E}\left[
\langle g_i^t,\log_{x_i^t}(x_\ast)\rangle
\right]
&\le \frac{1}{2}\Bigl(
\mathbb{E}[d^2(x_i^t,x_\ast)]
-\mathbb{E}[d^2(y_i^{t+1},x_\ast)]
\Bigr) \notag\\
&\quad + C_1\eta_t^2(\sigma^2+\delta^2).
\label{ineq:stochastic-thm-convex-1}
\end{align}
Since \(f_i\) is geodesically convex on \(\mathcal{X}\),
\begin{align*}
f_i(x_\ast)
\ge f_i(x_i^t)+\langle g_i^t,\log_{x_i^t}(x_\ast)\rangle,
\end{align*}
hence $-\eta_t \langle g_i^t,\log_{x_i^t}(x_\ast)\rangle
\ge \eta_t\bigl(f_i(x_i^t)-f_i(x_\ast)\bigr)$.
Combining this with \eqref{ineq:stochastic-thm-convex-1} and summing over
\(i=1,\dots,n\), we obtain
\begin{align}
&\eta_t \mathbb{E}[f(\mathbf{x}^t)-f(\mathbf{x}_\ast)] \le  C_1\eta_t^2(\sigma^2+\delta^2) + \\ &\frac{1}{2n}\sum_{i=1}^n
\Bigl(
\mathbb{E}[d^2(x_i^t,x_\ast)]
-\mathbb{E}[d^2(y_i^{t+1},x_\ast)]
\Bigr)  .
\label{ineq:stochastic-thm-convex-3}
\end{align}

Next, we analyze the consensus step. Applying Lemma~\ref{lemma:cosine-law} with
\(a=x_\ast\), \(b=y_i^{t+1}\), and \(c=x_i^{t+1}\), we get
\begin{align*}
&s\left\langle
\sum_{j=1}^n w_{ij}\log_{y_i^{t+1}}(y_j^{t+1}),
\log_{y_i^{t+1}}(x_\ast)
\right\rangle \\
&\le \frac{1}{2}\Bigl(
d^2(y_i^{t+1},x_\ast)-d^2(x_i^{t+1},x_\ast)
\Bigr) \\
&\quad + \frac{C_1}{2}s^2
\left\|\sum_{j=1}^n w_{ij}\log_{y_i^{t+1}}(y_j^{t+1})\right\|^2.
\end{align*}
Using Jensen's inequality on \(T_{y_i^{t+1}}\mathcal{M}\),
\begin{align*}
\left\|\sum_{j=1}^n w_{ij}\log_{y_i^{t+1}}(y_j^{t+1})\right\|^2
\le \sum_{j=1}^n w_{ij} d^2(y_i^{t+1},y_j^{t+1}),
\end{align*}
and hence
\begin{align}
&s\sum_{i=1}^n\sum_{j=1}^n w_{ij}
\left\langle
\log_{y_i^{t+1}}(y_j^{t+1}),
\log_{y_i^{t+1}}(x_\ast)
\right\rangle \notag\\
&\le \frac{1}{2}\Bigl(
\sum_{i=1}^n d^2(y_i^{t+1},x_\ast)
-\sum_{i=1}^n d^2(x_i^{t+1},x_\ast)
\Bigr) \notag\\
&\quad + \frac{C_1}{2}s^2
\sum_{i=1}^n\sum_{j=1}^n w_{ij} d^2(y_i^{t+1},y_j^{t+1}).
\label{ineq:stochastic-thm-convex-4}
\end{align}
On the other hand, applying the lower bound in
Lemma~\ref{lemma:cosine-law} with
\(a=x_\ast\), \(b=y_i^{t+1}\), and \(c=y_j^{t+1}\), we get
\begin{align*}
&\left\langle
\log_{y_i^{t+1}}(y_j^{t+1}),
\log_{y_i^{t+1}}(x_\ast)
\right\rangle
\ge \\ & \frac{1}{2}\Bigl(
d^2(y_i^{t+1},x_\ast)-d^2(y_j^{t+1},x_\ast)
\Bigr)  + \frac{C_2}{2}d^2(y_i^{t+1},y_j^{t+1}).
\end{align*}
Multiplying by \(w_{ij}\), summing over \(i,j\), and using that \(W\) is
doubly stochastic, the distance-to-\(x_\ast\) terms cancel:
\begin{align*}
&\sum_{i=1}^n\sum_{j=1}^n w_{ij}
\left\langle
\log_{y_i^{t+1}}(y_j^{t+1}),
\log_{y_i^{t+1}}(x_\ast)
\right\rangle \\
&\ge \frac{C_2}{2}
\sum_{i=1}^n\sum_{j=1}^n w_{ij} d^2(y_i^{t+1},y_j^{t+1}).
\end{align*}
Combining this with \eqref{ineq:stochastic-thm-convex-4} yields
\begin{align*}
&\frac{1}{2}\Bigl(
\sum_{i=1}^n d^2(y_i^{t+1},x_\ast)
-\sum_{i=1}^n d^2(x_i^{t+1},x_\ast)
\Bigr) \\
&\ge \frac{C_2s-C_1s^2}{2}
\sum_{i=1}^n\sum_{j=1}^n w_{ij} d^2(y_i^{t+1},y_j^{t+1}).
\end{align*}
Since \(s=C_2/(2C_1)\), we have
\(C_2s-C_1s^2=C_2^2/(4C_1)>0\). Therefore,
\begin{align}
\frac{1}{2n}\sum_{i=1}^n
\Bigl(
\mathbb{E}[d^2(y_i^{t+1},x_\ast)]
-\mathbb{E}[d^2(x_i^{t+1},x_\ast)]
\Bigr)
\ge 0.
\label{ineq:stochastic-thm-convex-5}
\end{align}
Plugging \eqref{ineq:stochastic-thm-convex-5} into
\eqref{ineq:stochastic-thm-convex-3}, we obtain
\begin{align*}
&\eta_t \mathbb{E}[f(\mathbf{x}^t)-f(\mathbf{x}_\ast)]
\le  \\ & \frac{1}{2n}\sum_{i=1}^n
\Bigl(
\mathbb{E}[d^2(x_i^t,x_\ast)]
-\mathbb{E}[d^2(x_i^{t+1},x_\ast)]
\Bigr) + C_1\eta_t^2(\sigma^2+\delta^2).
\end{align*}
Summing from \(t=1\) to \(T\), using
\(\sum_{t=1}^T \eta_t^2 = \eta_0^2\sum_{t=1}^T t^{-1}
\le \eta_0^2(1+\log T)\), gives
\begin{align*}
&\sum_{t=1}^T \eta_t \mathbb{E}[f(\mathbf{x}^t)-f(\mathbf{x}_\ast)] \le  C_1(\sigma^2+\delta^2)\sum_{t=1}^T \eta_t^2  \\
& +\frac{1}{2n}\sum_{i=1}^n
\Bigl(
\mathbb{E}[d^2(x_i^1,x_\ast)]
-\mathbb{E}[d^2(x_i^{T+1},x_\ast)]
\Bigr) \\
&\le \frac{1}{2n}\sum_{i=1}^n \mathbb{E}[d^2(x_i^1,x_\ast)]
+ \eta_0^2 C_1(\sigma^2+\delta^2)(1+\log T).
\end{align*}
Also,
\(\sum_{t=1}^T \eta_t = \eta_0\sum_{t=1}^T t^{-1/2}
\ge \eta_0\sqrt{T}\), and thus
\begin{align}
\frac{\sum_{t=1}^T \eta_t \mathbb{E}[f(\mathbf{x}^t)-f(\mathbf{x}_\ast)]}
{\sum_{t=1}^T \eta_t}
&\le \frac{1}{2n\eta_0\sqrt{T}}
\sum_{i=1}^n \mathbb{E}[d^2(x_i^1,x_\ast)] \notag\\
&\quad + \eta_0 C_1(\sigma^2+\delta^2)\frac{1+\log T}{\sqrt{T}}.
\label{ineq:fxt-fxast}
\end{align}

It remains to pass from \(\mathbf{x}^t\) to \(\bar{\mathbf{x}}^t\). Since each
\(f_i\) is \(\delta\)-Lipschitz,
\begin{align*}
f_i(\bar{x}^t)-f_i(x_\ast)
&= f_i(\bar{x}^t)-f_i(x_i^t)+f_i(x_i^t)-f_i(x_\ast) \\
&\le \delta\, d(\bar{x}^t,x_i^t)+f_i(x_i^t)-f_i(x_\ast).
\end{align*}
Summing over \(i\) and using Cauchy--Schwarz,
\begin{align*}
&f(\bar{\mathbf{x}}^t)-f(\mathbf{x}_\ast)
\le \frac{\delta}{n}\sum_{i=1}^n d(\bar{x}^t,x_i^t)
+ f(\mathbf{x}^t)-f(\mathbf{x}_\ast) \\
&\le \frac{\delta}{\sqrt{n}}
\left(\sum_{i=1}^n d^2(\bar{x}^t,x_i^t)\right)^{1/2}
+ f(\mathbf{x}^t)-f(\mathbf{x}_\ast).
\end{align*}
Taking expectations, using Jensen's inequality for the concave map
\(u\mapsto \sqrt{u}\), and applying Theorem~\ref{thm:stochastic-consensus},
we get
\begin{align}
&\mathbb{E}[f(\bar{\mathbf{x}}^t)-f(\mathbf{x}_\ast)]
\le \frac{\delta}{\sqrt{n}}
\left(
\mathbb{E}\left[\sum_{i=1}^n d^2(\bar{x}^t,x_i^t)\right]
\right)^{1/2} \\
& \qquad + \mathbb{E}[f(\mathbf{x}^t)-f(\mathbf{x}_\ast)] \notag\\
&\le \eta_0 \delta \sqrt{C(\xi)B}\,\frac{1}{\sqrt{t}}
+ \mathbb{E}[f(\mathbf{x}^t)-f(\mathbf{x}_\ast)].
\label{ineq:fbar-fxast}
\end{align}
Multiplying \eqref{ineq:fbar-fxast} by
\(\eta_t=\eta_0/\sqrt{t}\), summing from \(t=1\) to \(T\), dividing by
\(\sum_{t=1}^T \eta_t\), and using
\(\sum_{t=1}^T t^{-1}\le 1+\log T\) together with
\(\sum_{t=1}^T \eta_t\ge \eta_0\sqrt{T}\), we conclude
\begin{align*}
&\frac{\sum_{t=1}^T \eta_t \mathbb{E}[f(\bar{\mathbf{x}}^t)-f(\mathbf{x}_\ast)]}
{\sum_{t=1}^T \eta_t} \\
&\le
\frac{\eta_0^2 \delta \sqrt{C(\xi)B}\sum_{t=1}^T t^{-1}}
{\sum_{t=1}^T \eta_t}
+
\frac{\sum_{t=1}^T \eta_t \mathbb{E}[f(\mathbf{x}^t)-f(\mathbf{x}_\ast)]}
{\sum_{t=1}^T \eta_t} \\
&\le
\eta_0 \delta \sqrt{C(\xi)B}\,\frac{1+\log T}{\sqrt{T}}
+
\frac{1}{2n\eta_0\sqrt{T}}
\sum_{i=1}^n \mathbb{E}[d^2(x_i^1,x_\ast)] \\
&\quad +
\eta_0 C_1(\sigma^2+\delta^2)\frac{1+\log T}{\sqrt{T}},
\end{align*}
which proves the theorem.

\end{proof}
\vspace{-0.5cm}
\section{Numerical Analysis}\label{sec:numerics}
In this section, we evaluate the performance of the proposed (referred to as \texttt{Diffusion\_Dim}) on the task of Distributed Principal Component Analysis (PCA) \cite{wang2025riemannian}. We compare our approach against existing baselines, including the standard Riemannian Diffusion Adaptation (\texttt{Diffusion}) \cite{wang2025riemannian} and Decentralized Riemannian Stochastic Gradient Descent (\texttt{DRSGD}) \cite{chen2021decentralized}. 
\begin{figure}[ht!]
    \centering
\begin{subfigure}{.48\linewidth}
  \centering
  \includegraphics[width=.93\linewidth]{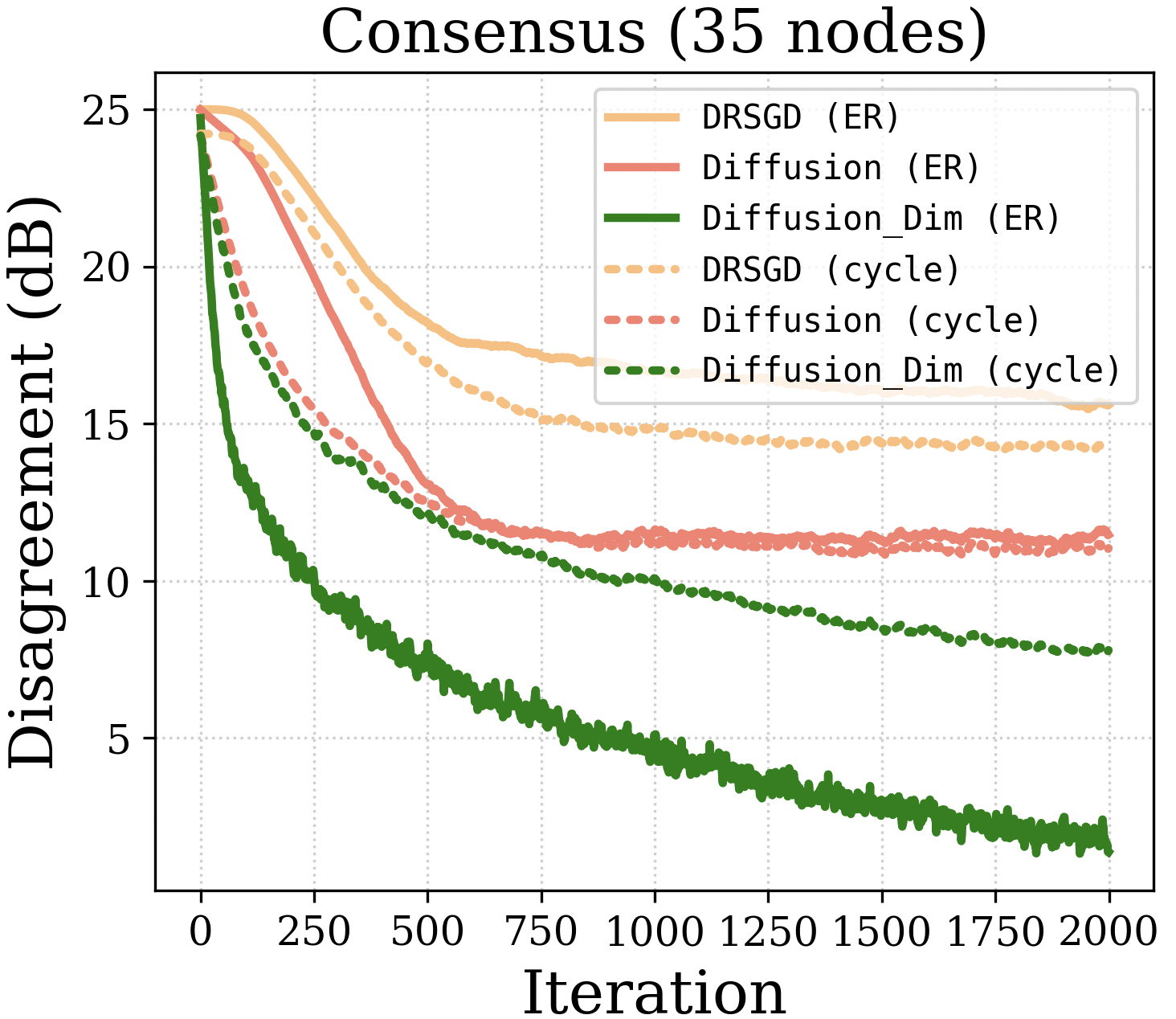}
  \label{fig:35-consensus}
\end{subfigure}%
\vspace{0.2cm}
\begin{subfigure}{.48\linewidth}
  \centering
  \includegraphics[width=.93\linewidth]{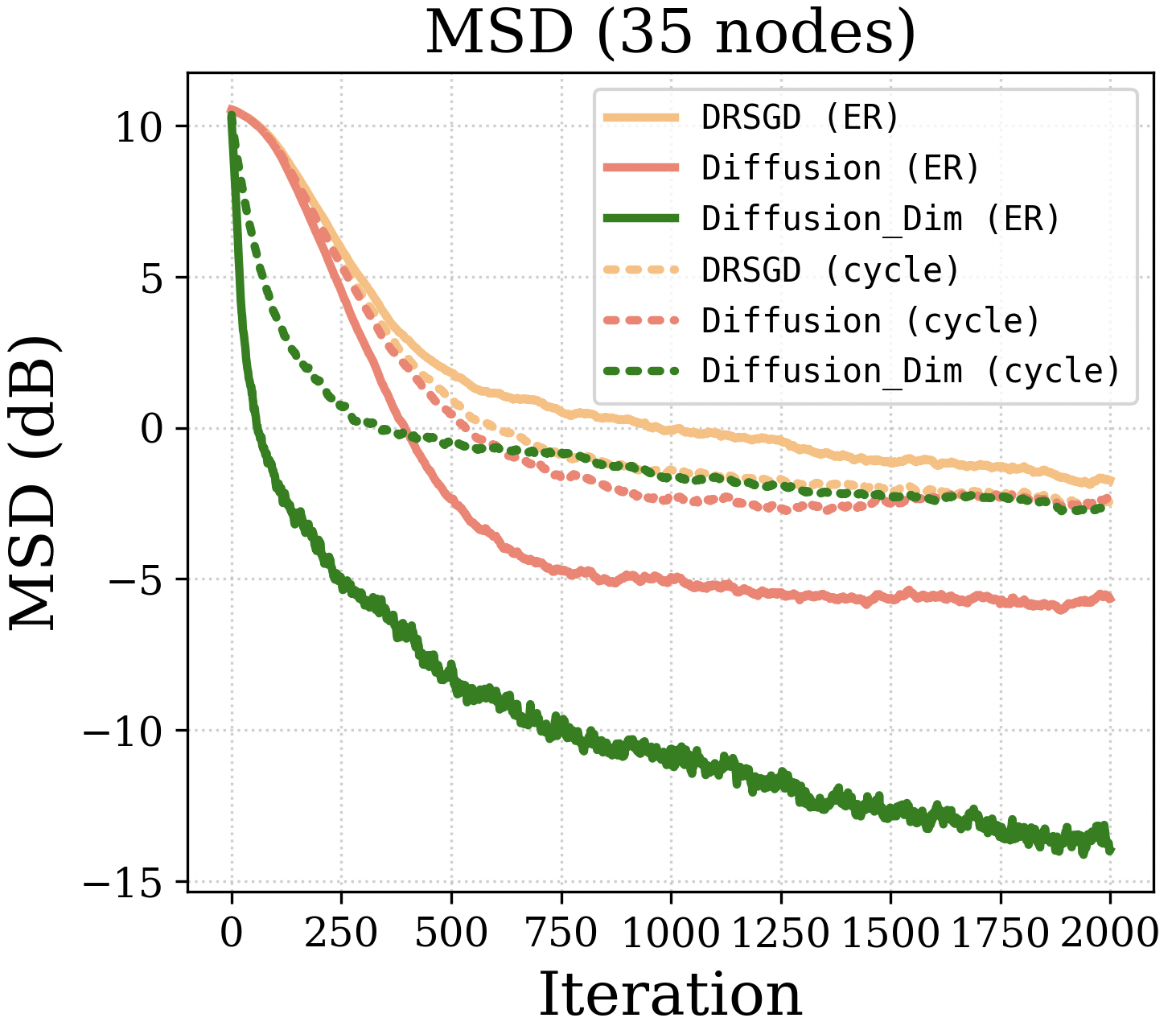}
  \label{fig:35-MSD}
\end{subfigure}\\
\begin{subfigure}{.48\linewidth}
  \centering
  \includegraphics[width=.93\linewidth]{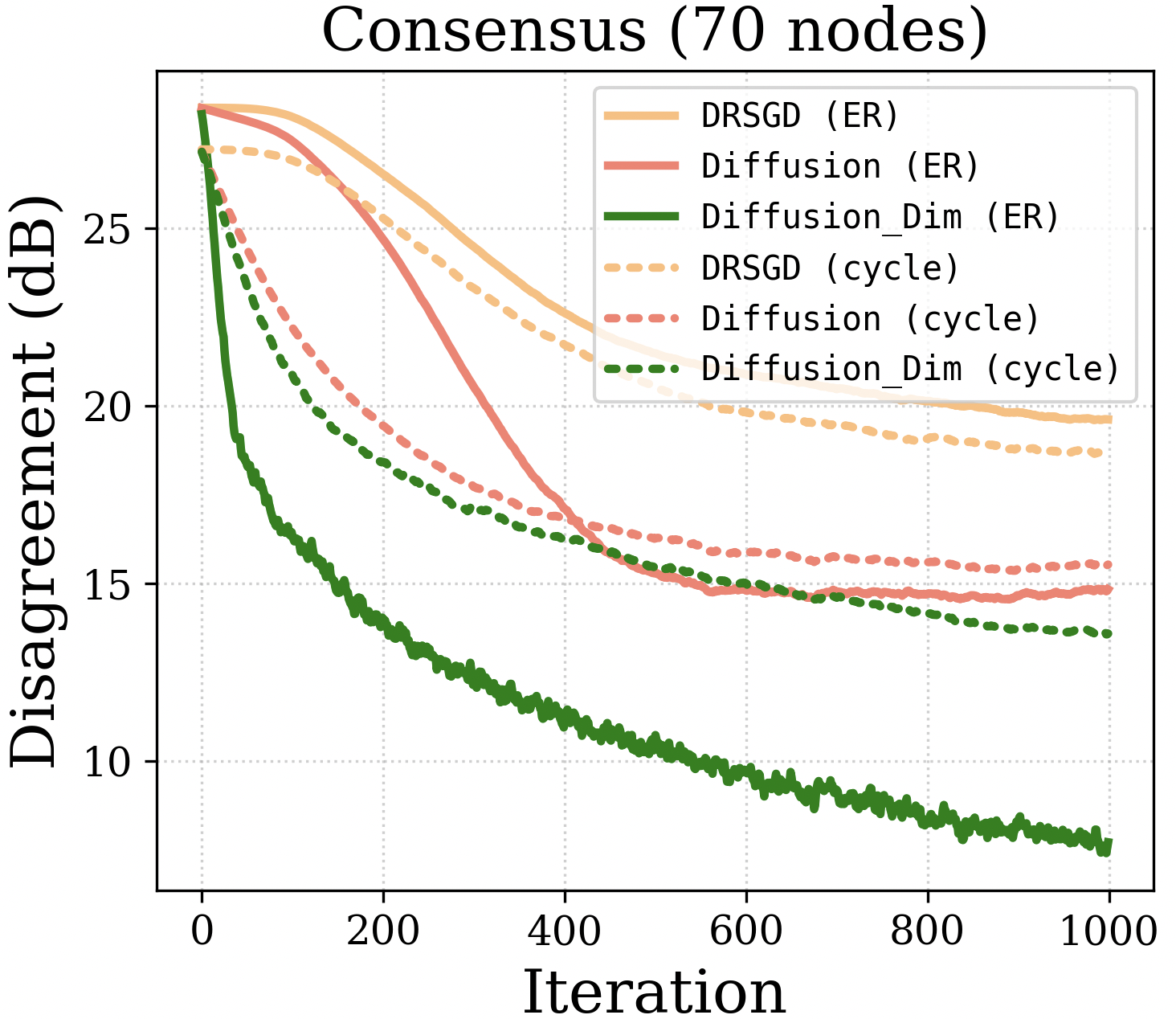}
  \label{fig:70-consensus}
\end{subfigure}%
\vspace{0.2cm}
\begin{subfigure}{.48\linewidth}
  \centering
  \includegraphics[width=.93\linewidth]{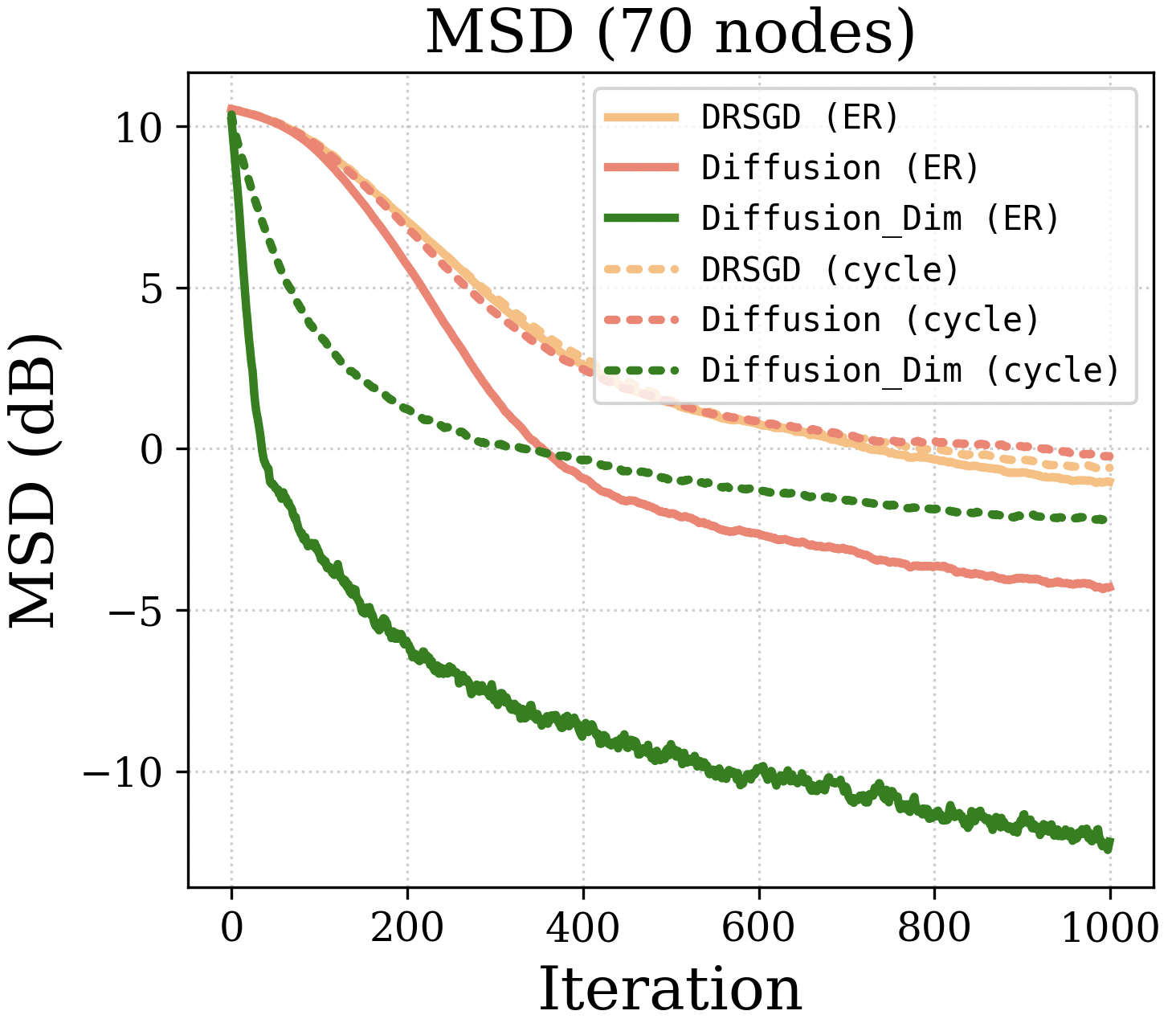}
  \label{fig:70-MSD}
\end{subfigure}\\
\begin{subfigure}{.48\linewidth}
  \centering
  \includegraphics[width=.93\linewidth]{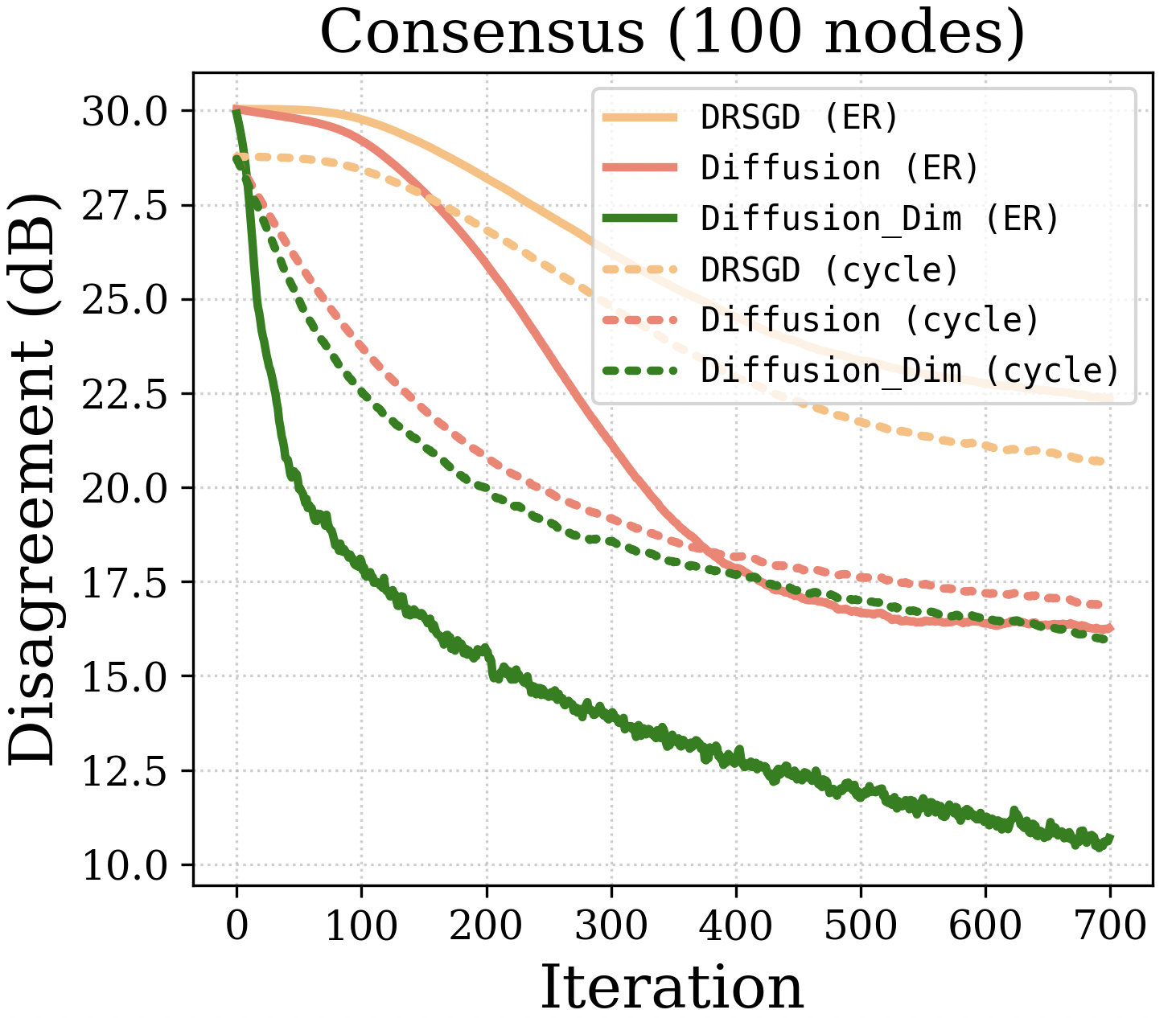}
  \label{fig:100-consensus}
\end{subfigure}%
\vspace{0.2cm}
\begin{subfigure}{.48\linewidth}
  \centering
  \includegraphics[width=.93\linewidth]{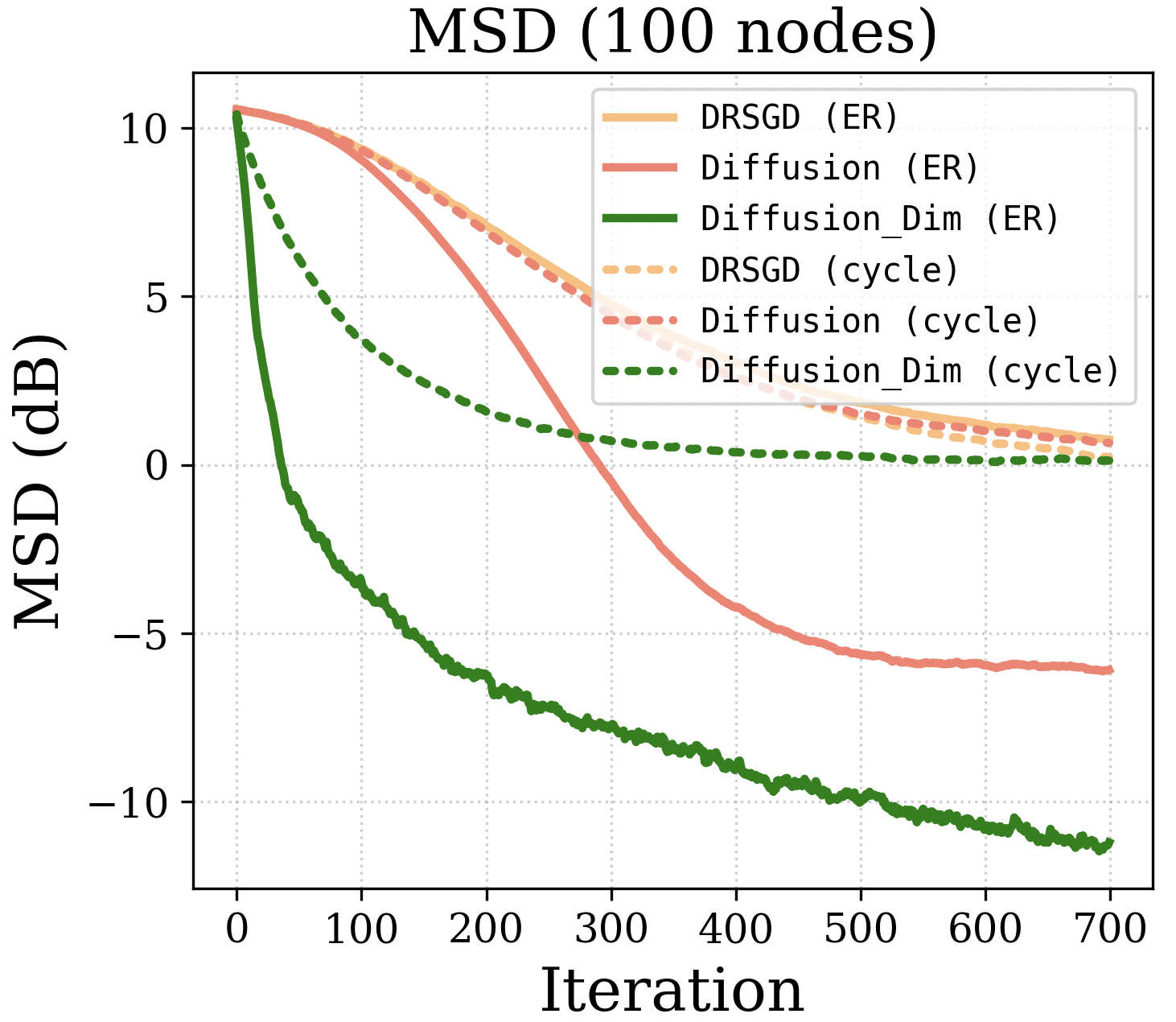}
  \label{fig:100-MSD}
\end{subfigure}
\vspace{-0.2cm}
    \caption{Consensus and Mean Squared Deviation (MSD) over ER and Cycle graphs with 35, 70, and 100 nodes }
    \label{fig:Consensus and MSD}
    \vspace{-0.7cm}
\end{figure}

\paragraph{Experiment set-up} The objective is to compute the first $p=5$ principal components of the MNIST dataset. The dataset consists of 60,000 handwritten digits, where each image is represented by $784$ pixels. The data is normalized to the range $[0, 1]$ and centered by subtracting the full dataset's mean. To benchmark the performance across varying network conditions, we consider six different graphs.
\begin{itemize}
    \item  Erdős-Rényi (ER) Graphs: With $n \in \{35, 70, 100\}$ nodes and an edge probability of $p=0.3$.
    \item Cycle Graphs: With $n \in \{35, 70, 100\}$ nodes.
\end{itemize}
The full dataset is randomly shuffled and partitioned equally among the $n$ agents. The problem is solved on the Grassmann manifold $\mathcal{G}_{784}^5$~\cite{edelman1998geometry}. The ``ground truth'' optimal subspace is computed by performing PCA on the entire 60,000-image dataset offline. Following~\cite{wang2025riemannian}, we adopt the corresponding fixed step sizes optimized for \texttt{Diffusion}: $\eta = 0.002$ and $s = 0.005$. For \texttt{Diffusion\_Dim} (Our Method), we utilize a diminishing step size $\eta_t = \eta_0 / \sqrt{t}$. For ER graphs, we use a significantly larger initial step size, $\eta_0 = 0.1$, for both the gradient and consensus step sizes, $s = 0.1$. For Cycle graphs, the initial step size $\eta_0 = 0.05$, which is moderately increased compared to the fixed baseline $\eta = 0.005$. We evaluate the algorithms using two metrics: Network Disagreement (Consensus Error), measured as $\sum w_{ij} d^2(x_i, x_j)$, and Mean Square Deviation (MSD), defined as the average squared geodesic distance to the global optimum \cite{wang2025riemannian}. Both metrics are reported in decibels (dB).

\paragraph{Performance on Erdős-Rényi Graphs} As shown in Fig. \ref{fig:Consensus and MSD} (solid lines), our \texttt{Diffusion\_Dim} method (green solid line) significantly outperforms all other decentralized baselines on ER graphs for both consensus error and MSD. Because ER graphs have high connectivity, they can tolerate the much larger initial step size ($\eta_0 = 0.1$). As the step size diminishes, the algorithm avoids the ``steady-state floor'' typically seen with fixed step sizes, reaching much lower MSD levels (approaching -15 dB) compared to the standard Diffusion method (-5 dB).
\paragraph{Performance on Cycle Graphs} On Cycle graphs (dashed lines), consensus is more difficult to achieve due to the large network diameter. However, \texttt{Diffusion\_Dim} (green dashed line) still provides a noticeable improvement over standard \texttt{Diffusion} and \texttt{DRSGD}. By slightly increasing the initial step size and then diminishing it, we achieve a faster initial descent and a more refined final estimate. 

\section{Discussion and Future Work}\label{sec:Discussion}
In this paper, we introduce a decentralized stochastic Riemannian optimization algorithm with a diminishing step size, thereby eliminating the steady-state bias of existing intrinsic-diffusion-type methods. In the static geodesically convex setting, we proved a non-asymptotic $\mathcal{O}(1/T)$ bound for the network consensus error and an $\mathcal{O}(\log T/\sqrt{T})$ ergodic bound for the global optimality gap. Promising future directions include extending the analysis to nonconvex objectives and developing intrinsic momentum-based methods.
\vspace{-0.1cm}
\bibliographystyle{IEEEtran}
\bibliography{references1}

\end{document}